\documentclass[oneside]{amsart}
\usepackage[T1]{fontenc}		
\usepackage{braha-packages}		
\usepackage{afterpage}
\usepackage{tikz}				
\usepackage[font=small,labelfont=bf]{caption}
\usepackage{braha-refs}			

\usepackage{braha-thms}			
\usepackage{braha-commands}		

\usepackage[abbrev]{amsrefs}			
\renewcommand{\MR}[1]{ \href{https://mathscinet.ams.org/mathscinet-getitem?mr=MR#1}{MR#1}}	
\newcommand{\Zbl}[1]{ \href{https://zbmath.org/?q=an:#1}{Zbl~#1}}	
\newcommand{\JFM}[1]{ \href{https://zbmath.org/?q=an:#1}{JFM~#1}}	
\newcommand{\arxiv}[1]{\href{https://arxiv.org/abs/#1}{{\tt arXiv:#1}}}	

\usetikzlibrary{decorations.pathmorphing}
\usetikzlibrary{decorations.pathreplacing,decorations.markings}
\tikzset{
  on each segment/.style={
    decorate,
    decoration={
      show path construction,
      moveto code={},
      lineto code={
        \path [#1]
        (\tikzinputsegmentfirst) -- (\tikzinputsegmentlast);
      },
      curveto code={
        \path [#1] (\tikzinputsegmentfirst)
        .. controls
        (\tikzinputsegmentsupporta) and (\tikzinputsegmentsupportb)
        ..
        (\tikzinputsegmentlast);
      },
      closepath code={
        \path [#1]
        (\tikzinputsegmentfirst) -- (\tikzinputsegmentlast);
      },
    },
  },
  mid arrow/.style={postaction={decorate,decoration={
        markings,
        mark=at position .5 with {\arrow[#1]{stealth}}
      }}},
}

\makeatletter
\@namedef{subjclassname@2020}{\textup{2020} Mathematics Subject Classification}
\makeatother

\title[Choice vs.\ chromatic number for graphs on orientable surfaces]%
{The choice number versus the chromatic number for graphs embeddable on orientable surfaces}

\author[N. Balachandran]{Niranjan Balachandran}
\address{Department of Mathematics, Indian Institute of Technology Bombay, Powai, Mumbai 400076}
\email{niranj@math.iitb.ac.in}

\author[B. Sankarnarayanan]{Brahadeesh Sankarnarayanan}
\address{Department of Mathematics, Indian Institute of Technology Bombay, Powai, Mumbai 400076}
\email{bs@math.iitb.ac.in}
\thanks{Research of Brahadeesh Sankarnarayanan was supported by the National Board for Higher Mathematics (NBHM),
Department of Atomic Energy (DAE), Govt.\ of India.}


\date{\today}

\begin{document}

\begin{abstract}
We show that for loopless \(6\)-regular triangulations on the torus the gap between the choice
number and chromatic number is at most \(2\). We also show that the largest gap for
graphs embeddable in an orientable surface of genus \(g\) is of the order \(\Theta(\sqrt{g})\),
and moreover for graphs with chromatic number of the order \(o(\sqrt{g}/\log_{2}(g))\) the largest gap
is of the order \(o(\sqrt{g})\).
\end{abstract}

\maketitle

\section{Introduction}\label{S:Introduction}

We shall denote by \(\Nat\) the set of natural numbers \(\Set{ 0, 1, 2, \dotsc }\).
For a graph \(G = (V, E)\), the notation \(\diredge{v}{w}\)
shall indicate that the edge \(\edge{v}{w} \in E\) is oriented from the vertex \(v\)
to the vertex \(w\). A graph \(G\) is \defining{directed} if every edge of \(G\)
is oriented. All logarithms in this paper are to the base \(2\).
We freely make use of the Bachmann--Landau--Knuth
notations, but state them briefly for completeness below.
Let \(f\) and \(g\) be positive functions of real variables.
(1)~\(f = \littleoh{g}\) if \(\lim_{x \to \infty} f(x)/g(x) = 0\);
(2)~\(f \leq \bigoh{g}\) if there is a constant \(M > 0\) such that \(\limsup_{x \to \infty} f(x)/g(x) \leq M\);
(3)~\(f \geq \bigomega{g}\) if \(g \leq \bigoh{f}\); and
(4)~\(f = \bigtheta{g}\) if \(\bigomega{g} \leq f \leq \bigoh{g}\).

A \defining{(vertex) coloring} of a graph \(G = (V, E)\)
is an assignment of a ``color'' to each vertex, that is, an
assignment \(v \mapsto \Color(v) \in \Nat\) for every \(v \in V(G)\).
A coloring of \(G\) is \defining{proper} if adjacent vertices
receive distinct colors.
\(G\) is \defining{\(k\)-colorable} if
there exists a proper coloring of the vertices using at most \(k\) colors.
The least integer \(k\) for which \(G\) is \(k\)-colorable is called the
\defining{chromatic number} of \(G\) and is denoted \(\chr(G)\).
If \(\chr(G) = k\), we also say that \(G\) is \defining{\(k\)-chromatic}.

A variation of \(k\)-colorability called \defining{\(k\)-choosability}
was defined independently by Vizing~\cite{Viz76} in 1976
and Erd{\H o}s, Rubin and Taylor~\cite{ErdRubTay80} in 1979.
A \defining{list assignment} \(\ListAsgn{L}\) on \(G\) is a
collection of sets of the form \(\ListAsgn{L} = \Set{ \List{L}{v} \subset \Nat \given v \in V(G) }\),
where one thinks of each \(\List{L}{v}\) as a \defining{list} of colors available for coloring the vertex
\(v \in V(G)\). \(G\) is \defining{\(\ListAsgn{L}\)-choosable}
if there exists a proper coloring of the vertices such that \(\Color(v) \in \List{L}{v}\)
for every \(v \in V(G)\). \(G\) is \defining{\(k\)-choosable}
if it is \(\ListAsgn{L}\)-choosable
for every list assignment \(\ListAsgn{L}\) with \(\card{\List{L}{v}} \leq k\)
for all \(v \in V(G)\). The least
integer \(k\) for which \(G\) is \(k\)-choosable is called the \defining{choice number}
or \defining{list chromatic number} of \(G\) and is denoted \(\ch(G)\). If \(\ch(G) = k\),
we also say that \(G\) is \defining{\(k\)-list-chromatic}.

The choice number generalizes the chromatic number in the following sense:
if \(\ListAsgn{L}\) is a list assignment in which all the lists
\(\List{L}{v}\) are identical and have cardinality at most \(k\),
then \(G\) is \(\ListAsgn{L}\)-choosable if and only if \(G\) is \(k\)-colorable.
Hence, \(\chr(G) \leq \ch(G)\) for any graph \(G\). In general, however,
\(\chr(G) < \ch(G)\), so it behooves one
to investigate the nature of the gap between the choice number and chromatic number.

One line of investigation is to examine \defining{chromatic-choosable} graphs,
that is, graphs that satisfy \(\chr(G) = \ch(G)\).
We mention one of the important results
concerning such graphs, conjectured by Ohba~\cite{Ohb02} in 2002
and subsequently settled in the affirmative by Noel, Reed and Wu~\cite{NoeReeWu15} in 2015:
if \(G\) is a graph on at most \(2\chr(G) + 1\) vertices, then it satisfies \(\chr(G) = \ch(G)\).

The opposite line of investigation is to examine the width of the gap between the
chromatic number and choice number. In \cite{ErdRubTay80}, Erd{\H o}s, Rubin and Taylor
showed that there are bipartite graphs (that is, graphs with \(\chr(G) = 2\))
that have arbitrarily large choice number;
more precisely, they showed that \(\ch(K_{n, n}) > k\) if \(n \geq \binom{2k-1}{k}\).
At first glance it appears that this line of investigation is thus fruitless,
but one has to note that the graphs \(K_{n, n}\) have high average degree.
In fact, Alon~\cite{Alo00} showed in 2000 
that \(\ch(G) \geq \parens[\big]{\frac{1}{2}-\littleoh{1}}\log(\delta)\),
where \(\delta \equiv \delta(G)\) is the minimum degree of \(G\).

Thus, one is motivated to bound the minimum degree of graphs
in order to examine the gap between the chromatic
number and choice number. A natural criterion for doing so
is to consider graphs that are embeddable in a fixed surface.
By a \defining{surface} we mean a compact connected
\(2\)-manifold. Informally, a graph \(G = (V, E)\) is \defining{embeddable}
in a surface if there exists a drawing of \(G\) on the surface
without any crossing edges. By the classification of surfaces theorem,
every orientable surface is homeomorphic to a sphere with \(g \geq 0\)
handles, denoted \(S_{g}\), and every nonorientable surface is homeomorphic to
a sphere with \(k \geq 1\) crosscaps, denoted \(N_{k}\). The \defining{genus}
of the surface \(S_{g}\) (resp.\ \(N_{k}\)) is defined to be \(g\)
(resp.\ \(k\)). We shall primarily restrict our attention to graphs embeddable on orientable surfaces
in what follows.

Now, suppose that \(G = (V, E)\) is a connected graph
which is embeddable in \(S_{g}\), \(g \geq 0\). Choose an embedding
and denote by \(\numv\), \(\nume\) and \(\numf\) the
number of vertices, edges and faces of \(G\),
respectively, in this embedding.
Euler's formula says that \(\numv - \nume + \numf \geq 2 - 2g\),
with equality holding if every face is homeomorphic to a disc.
Let \(F \equiv F(G)\) be the set of faces of \(G\) in this embedding.
If \(\degree(f) \geq 3\) for each \(f \in F\), then we have \(3\numf \leq 2\nume\),
and this bounds the minimum degree as \(\delta(G) \leq 2\nume/\numv \leq 6 + 12(g - 1)/\numv\).
Hence, \(\delta(G) \leq 5\) if \(g = 0\) and \(\delta(G) \leq 12g + 6\) if \(g \geq 1\).
Thus, we make the following definitions:
\begin{definition}
	For a graph \(G = (V, E)\), define the \defining{jump of \(G\)}
	by \(\jump(G) \defn \ch(G) - \chr(G)\). For each \(g \geq 0\),
	define the \defining{jump at \(g\)} by
	\[
		\jump(g) \defn \max\Set{ \jump(G) \given G \text{ is embeddable in } S_{g} }.
	\]	
\end{definition}

For graphs embeddable in the sphere \(S_{0}\)
(equivalently, for planar graphs), the investigation of \(\jump(0)\)
was indicated by Erd{\H o}s, Rubin and Taylor~\cite{ErdRubTay80}
through the following conjectures and question:
\begin{itemize}[font=\bfseries]
	\item[(C1)]\label{C1} Every planar graph is \(5\)-choosable.
	
	\item[(C2)]\label{C2} There exists a planar graph that is not \(4\)-choosable.
	
	\item[(Q)]\label{Q} Does there exist
	a planar bipartite graph that is not \(3\)-choosable?
\end{itemize}
Alon and Tarsi~\cite{AloTar92} in 1992 answered
\textbf{(Q)} in the negative by showing that every planar bipartite graph is \(3\)-choosable;
this is also best possible since there are simple examples~\cite{ErdRubTay80} of planar bipartite
graphs that are not \(2\)-choosable.
Voigt~\cite{Voi93} in 1993 settled \textbf{(C2)} positively 
by constructing a planar graph on 238 vertices that is not \(4\)-choosable,
and Thomassen~\cite{Tho94b} in 1994 settled \textbf{(C1)} positively 
through a remarkably short and elegant proof.

Thus, \(\jump(0) \leq 2\), with equality holding if and only if
there exists a planar \(3\)-chromatic graph that is not \(4\)-choosable.
Mirzakhani~\cite{Mir96} in 1996 constructed such a graph on 63 vertices;
parallelly, Voigt and Wirth~\cite{VoiWir97} in 1997
observed that a non-\(4\)-choosable planar graph on 75 vertices constructed
by Gutner~\cite{Gut96} in 1996 is \(3\)-chromatic.
Thus, it is known that \(\jump(0) = 2\).

We pose the analogous question for toroidal graphs:
\begin{question}\label{Q:toroidal-bound}
	What is \(\jump(1)\)?
	That is, how large can the gap between the choice number and chromatic number be
	for a toroidal graph?
\end{question}
Since every planar graph is also toroidal, the maximum gap
for toroidal graphs cannot be smaller than \(2\).
In \cref{S:jump-triangulation}, we examine \(6\)-regular triangulations on the torus and
show the following result:
\begin{theorem}\label{T:Main}
	For any loopless \(6\)-regular triangulation \(G\) on the torus, \(\jump(G) \leq 2\).
\end{theorem}

While computing \(\jump(g)\) precisely for larger values of \(g\)
seems difficult, we are able to describe the
asymptotic behavior of \(\jump(g)\) in \cref{S:asymptotics}
as follows:
\begin{theorem}\label{T:asymp1}
	\(\jump(g) = \bigtheta{\sqrt{g}}\). That is, there exist two positive constants \(c_{1}\) and \(c_{2}\) such that
	\[
		c_{1} \sqrt{g} \leq \jump(g) \leq c_{2} \sqrt{g}
	\]
	for all sufficiently large \(g\).
\end{theorem}

A natural follow-up is to investigate for which graphs
this largest gap is attained, for which the
following classical result is useful:
\begin{theorem}[Heawood~\cite{Hea90}, 1890]\label{T:Heawood}
	Let \(g \geq 1\). If \(G\) is embeddable in \(S_{g}\), then
	\[
		\chr(G) \leq \ch(G) \leq H(g) \defn \floor*{\frac{7 + \sqrt{1 + 48g}}{2}}.
	\]
\end{theorem}
Heawood proved that \(H(g)\), the so called \defining{Heawood number},
is an upper bound for \(\chr(G)\), and essentially the same
argument carries forward to prove that \(H(g)\) is an upper bound for
\(\ch(G)\) as well. Now, \Cref{T:Heawood} shows that
if \(G\) is \(H(g)\)-chromatic then it is also \(H(g)\)-list-chromatic,
so such graphs can never attain the maximum gap for \(S_{g}\).
The same is true if \(G\) is \((H(g)-1)\)-chromatic, by the following
result known as Dirac's map color theorem:
\begin{theorem}[Dirac~\cite{Dir52c}, 1952, B{\"o}hme--Mohar--Stiebitz~\cite{BohMohSti99}, 1999]\label{T:Dirac}
	If \(G\) is embeddable in \(S_{g}\) with \(\chr(G) = H(g)\) or \(\ch(G) = H(g)\), then
	\(K_{H(g)}\) is a subgraph of \(G\).
\end{theorem}
Dirac proved this result for the chromatic number,
and it was later extended to the choice number by B{\"o}hme, Mohar and Stiebitz.

At the other end, consider the complete bipartite graph \(K_{n, n}\).
It is an easy exercise that \(K_{n, n}\) is \(k\)-choosable where
\(k \defn \floor{\log(n)} + 1\), and it is known that \(K_{m, n}\) is embeddable
in \(S_{g}\) for \(g \defn \ceil{(m - 2)(n - 2) / 4}\) and this is best possible (see~\cite{Rin65}).
Hence,
\[
	\jump(K_{n, n}) \leq \log(n) - 1 \leq \log\parens[\big]{2\sqrt{g + 1} + 2} - 1 \leq \bigoh[\big]{\log(g)} \leq \littleoh{\sqrt{g}}.
\]
Since any bipartite graph is a subgraph of a complete bipartite graph,
this shows that one does not expect bipartite graphs to attain the greatest
gap on a fixed surface.

This motivates the following definition:
\begin{definition}
	For each \(g \geq 0\), \(r \geq 1\), define
	\[
		\jump(g, r) \defn \max\Set{ \jump(G) \given G \text{ is connected and embeddable in } S_{g}, \chr(G) = r }
	\]
	whenever the set on the right is nonempty. If there is no connected graph embeddable in \(S_{g}\)
	having chromatic number \(r\), then define \(\jump(g, r) \defn 0\).
\end{definition}

In \cref{S:asymptotics}, we prove the following stronger result along the
same lines as the bipartite case:
\begin{theorem}\label{T:asymp2}
	\(\jump(g, r) = \littleoh{\sqrt{g}}\) when \(r = \littleoh[\big]{\sqrt{g}/\log(g)}\).
	That is, if for each \(\delta > 0\) we have \(r \leq \delta \sqrt{g}/\log(g)\) for all sufficiently large \(g\),
	then for every \(\epsilon > 0\), \(\jump(g, r) \leq \epsilon \sqrt{g}\) for all sufficiently large \(g\).
\end{theorem}
 
The rest of the paper is organised as follows. In \cref{S:preliminaries},
we state some preliminary results that will be used in our proofs later on.
In \cref{S:jump-triangulation}, we prove \cref{T:Main}.
In \cref{S:asymptotics}, we prove \cref{T:asymp1,T:asymp2}.
In \cref{S:conclusion}, we mention some partial results
towards computing \(\jump(1)\), that is, the maximum gap \(\ch(G) - \chr(G)\) for toroidal
\(G\). We also generalise the definitions of \(\jump(g)\) and \(\jump(g, r)\),
as well as the results proved in \cref{S:asymptotics},
to graphs embeddable on nonorientable surfaces.
We conclude the paper with some open questions for further investigation.

\section{Preliminaries}\label{S:preliminaries}

We begin with the following result on the embeddability of
the complete graph \(K_{r}\) into an orientable surface:
\begin{theorem}[Ringel--Youngs~\cite{RinYou68}, 1968]\label{T:Ringel-Youngs}
	For every \(r \geq 1\), \(K_{r}\) is embeddable in \(S_{\OGenus(r)}\) for \(\OGenus(r) = \ceil{(r - 3) (r - 4) / 12}\),
	and this is best possible.
\end{theorem}
The above result, called the Ringel--Youngs theorem, also shows that
\(K_{H(g)}\) is embeddable in \(S_{g}\) for all \(g \geq 1\),
so it implies that Heawood's upper bound is tight for all \(g \geq 1\).

\begin{definition}
	Let \(G = (V, E)\) be a simple graph, and \(k \geq 1\). The \defining{\(k\)-core} of \(G\)
	is the unique maximal subgraph of \(G\) having minimum degree
	at least \(k\).
\end{definition}

The next result is folklore; it was observed for \(k = 2\) by Erd{\H o}s, Rubin and Taylor~\cite{ErdRubTay80},
but we include the proof for general \(k\) for the sake of completeness.
\begin{proposition}\label{T:core}
	Let \(G = (V, E)\) be a simple graph, and \(k \geq 1\). Then,
	\(G\) is \(k\)-choosable if and only if its \(k\)-core is \(k\)-choosable.
\end{proposition}
\begin{proof}
	If \(G\) is \(k\)-choosable, then so is every subgraph. In particular,
	its \(k\)-core is \(k\)-choosable. This proves the forward direction.
	For the reverse direction, suppose for the sake of contradiction
	that the \(k\)-core of \(G\) is \(k\)-choosable, but \(G\) is not \(k\)-choosable.
	Let \(H \leq G\) be a subgraph of \(G\) with the least order such that
	\(H\) is not \(k\)-choosable. Note that \(H\) must be nonempty.
	
	We show that the minimum degree of \(H\) must be at least \(k\). Suppose
	for the sake of contradiction that there exists
	a vertex \(v \in V(H)\) such that \(\degree(v) < k\).
	Let \(\ListAsgn{L}\) be a list assignment on \(G\) with lists of size \(k\)
	for which \(H\) is not \(\ListAsgn{L}\)-choosable.
	The subgraph \(H - v\) is \(k\)-choosable by the minimality of \(H\) with respect
	to the number of vertices, so properly color \(H - v\) using the list assignment \(\ListAsgn{L}\).
	Now, observe that the list \(\List{L}{v}\) on the vertex \(v\) has lost at most \(k - 1\)
	colors due to the coloring of its neighbors, hence the coloring on \(H - v\) extends to
	a proper coloring on \(H\). Hence, \(H\) is \(\ListAsgn{L}\)-choosable, a contradiction.

	Thus, the minimum degree of \(H\) is at least \(k\), and so \(H\) is a subgraph
	of the \(k\)-core of \(G\). But the \(k\)-core of \(G\) is \(k\)-choosable by assumption,
	so \(H\) is also \(k\)-choosable, a contradiction. This completes the proof.
\end{proof}

The next result, called Brooks's theorem, gives a useful bound on
the chromatic and choice numbers of a graph in terms of its maximum degree.

\begin{theorem}[Brooks~\cite{Bro41}, 1941, Vizing~\cite{Viz76}, 1976, Erd{\H o}s--Rubin--Taylor~\cite{ErdRubTay80}, 1979]\label{T:Brooks}
	Let \(G\) be a connected simple graph with maximum degree \(\Delta\).
	Then, \(G\) is \(\Delta\)-colorable (resp.\ \(\Delta\)-choosable), unless \(G\) is an odd cycle
	or the complete graph on \(\Delta + 1\) vertices, in which
	cases \(G\) is \((\Delta + 1)\)-chromatic (resp.\ (\(\Delta + 1\))-list-chromatic).
\end{theorem}

Brooks proved his result for the chromatic number, and this was later extended
to the choice number independently by Vizing and by Erd{\H o}s, Rubin and Taylor.

Another useful result in computing the choice number of graphs
is as follows.
\begin{definition}
	For a directed graph \(G\), a subgraph \(H\) is
	\defining{Eulerian} if \(\indegree_{H}(v) = \outdegree_{H}(v)\)
	for all \(v \in V(H)\). An \defining{even} (resp.\ \defining{odd})
	Eulerian subgraph is one with an even (resp.\ odd) number of edges.
\end{definition}
Note that the subgraphs are not assumed to be connected in the above definition.

\begin{theorem}[Alon--Tarsi~\cite{AloTar92}, 1992]\label{T:Alon--Tarsi}
	Suppose the edges of a graph \(G\) can be oriented such that
	the number of even Eulerian subgraphs differs from the number of
	odd Eulerian subgraphs. Let \(\ListAsgn{L}\) be a list assignment
	such that \(\card{L_{v}} \geq \outdegree(v) + 1\) for all \(v \in V(G)\).
	Then, \(G\) is \(\ListAsgn{L}\)-choosable.
\end{theorem}

\begin{notation}
	For every \(m, r \geq 1\), denote by
	\(K_{m * r}\) the complete \(r\)-partite graph with \(m\) vertices in each part.
\end{notation}

We will need the following bound
on the growth of the choice number of \(K_{m * r}\):
\begin{theorem}[Alon~\cite{Alo92}, 1992]\label{T:Alon-multipartite}
	There exist two positive constants \(c_{1}\) and \(c_{2}\) such that,
	for every \(m \geq 2\) and for every \(r \geq 2\),
	\[
		c_{1} r \log(m) \leq \ch(K_{m * r}) \leq c_{2} r \log(m).
	\]
\end{theorem}

If \(G = (V, E)\) is a \(k\)-regular triangulation on the torus \(S_{1}\),
then \(k \numv = 2 \nume\) and \(3\numf = 2\nume\), which together
with Euler's formula \(\numv - \nume + \numf = 0\) for \(S_{1}\)
gives \(k = 6\). So, any regular triangulation on the torus
is necessarily \(6\)-regular.
In 1973, Altshuler~\cite{Alt73} characterized the
\(6\)-regular triangulations on the torus as follows
(similar constructions can also be found in \cites{Neg83,Tho91}).
For integers \(r \geq 1\), \(s \geq 1\), and \(0 \leq t < s\),
let the graph \(G \defn T(r, s, t)\) have the vertex set
\(V(G) \defn \Set{ (i, j) \given 1 \leq i \leq r, 1 \leq j \leq s }\)
and the following edges:
\begin{itemize}
	\item For \(1 < i < r\), the vertex \((i, j)\) is adjacent to
	\((i, j \pm 1)\), \((i \pm 1, j)\) and
	\((i \pm 1, j \mp 1)\).
	
	\item If \(r > 1\), the vertex \((1, j)\) is adjacent to \((1, j \pm 1)\), \((2, j)\),
	\((2, j - 1)\), \((r, j + t + 1)\) and \((r, j + t)\). 
	
	\item If \(r > 1\), the vertex \((r, j)\) is adjacent to \((r, j \pm 1)\),
	\((r - 1, j + 1)\), \((r - 1, j)\), \((1, j - t)\) and \((1, j - t - 1)\).
	
	\item If \(r = 1\), the vertex \((1, j)\) is adjacent to \((1, j \pm 1)\), \((1, j \pm t)\)
	and \((1, j \pm (t + 1))\).
\end{itemize}
Here, addition in the second coordinate is taken modulo \(s\).

\(T(r, s, t)\) can be visualized by taking an \((r + 1) \times (s + 1)\) grid
graph, identifying the top and bottom rows in the usual manner,
identifying the leftmost and rightmost columns with a ``twist'' by \(t\) vertices,
and then triangulating each face in a regular manner.
\Cref{F:T(562)} shows the graph \(G = T(5, 6, 2)\).
The edges between the top and bottom rows are not shown in the figure.
Note that the graph in \cref{F:T(562)} happens to be \(3\)-chromatic.

\begin{theorem}[Altshuler~\cite{Alt73}, 1973]\label{T:Altshuler}
	Every \(6\)-regular triangulation on the torus is isomorphic
	to \(T(r, s, t)\) for some integers \(r \geq 1\), \(s \geq 1\), and \(0 \leq t < s\).
\end{theorem}

Note that \(T(r, s, t)\) can be isomorphic to \(T(r', s', t')\)
even if the tuples \((r, s, t)\) and \((r', s', t')\) are distinct.

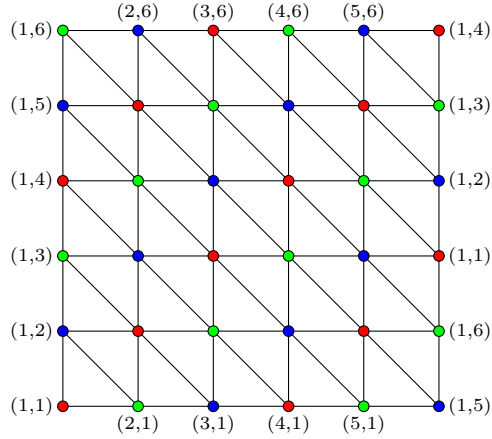
\begin{figure}
	\begin{tikzpicture}[font=\scriptsize]
	\draw (1,1) grid (6,6);
	\foreach \y in {2,...,6}
		\foreach \x in {1,...,5}{
			\draw (\x,\y) -- (\x+1,\y-1);
	}
	\foreach \x in {1,...,6}
		\foreach \y in {1,...,6}{
		\pgfmathparse{Mod(\y-\x,3)==0?1:0}
		\ifnum\pgfmathresult>0
			\draw[fill=red] (\x,\y) circle(2pt);
		\else
			\pgfmathparse{Mod(\y-\x,3)==1?1:0}
			\ifnum\pgfmathresult>0
				\draw[fill=blue] (\x,\y) circle(2pt);
			\else
				\draw[fill=green] (\x,\y) circle(2pt);
			\fi
		\fi
		}
	\foreach \y in {1,...,6}{
		\node[anchor=east] at (1,\y) {(1,\y)};
		\pgfmathparse{Mod(\y-5,6)}
		\node[anchor=west] at (6,\pgfmathresult+1) {(1,\y)};
	}
	\foreach \x in {2,...,5}{
		\node[anchor=north] at (\x,1) {(\x,1)};
		\node[anchor=south] at (\x,6) {(\x,6)};
	}
	\end{tikzpicture}
	\caption{\(G = T(5,6,2)\)}\label{F:T(562)}
\end{figure}

\section{The gap for \texorpdfstring{\(6\)}{6}-regular triangulations on the torus}\label{S:jump-triangulation}

We start by observing that some of the graphs \(T(r, s, t)\) are not simple triangulations:
they may contain loops and multiple edges.

The graphs containing loops are precisely those isomorphic to
\(T(1, s, 0)\) for all \(s \geq 1\), so we do not consider them. 
The loopless graphs containing multiple edges are precisely
those isomorphic to
\(T(2, s, t)\) for \(t = 0, s - 2, s - 1\), \(s \geq 2\), and \(T(1, s, t)\)
for \(t = 1, \floor{(s - 1) / 2}\), \(s \geq 3\). One can check that these graphs
are never bipartite, so the chromatic number of any such graph is at least \(3\).
Moreover, after deleting the duplicate edges (since they do not make
a difference for the purpose of coloring) these graphs have maximum degree
\(\Delta \leq 5\). Hence, by \cref{T:Brooks},
any such graph is \(5\)-choosable and thus has gap at most \(2\),
unless it is isomorphic to \(K_{6}\); but in the latter case,
\cref{T:Brooks} says that the chromatic number and choice number are both equal to
\(6\), so the gap is \(0\). Hence, \(\jump\parens[\big]{T(r, s, t)} \leq 2\)
if \(T(r, s, t)\) contains multiple edges. 

This takes care of the trivial cases. Next, we examine the graphs
\(T(r, s, t)\) that are simple.
First, we establish some notations and definitions for use in the upcoming proofs:
\begin{notation}
\hfill
	\begin{itemize}
	\item A list \(\List{L}{v}\) is called a \defining{\(k\)-list} if \(\card{\List{L}{v}} = k\).
	A \defining{\(k\)-list assignment \(\ListAsgn{L}\)} on a graph \(G\) is a list assignment
	for which every list \(L_{v}\) is a \(k\)-list. Note that a graph \(G\)
	is \(k\)-choosable if and only if it is \(\ListAsgn{L}\)-choosable for every
	\(k\)-list assignment \(\ListAsgn{L}\).
	
	\item Let \(G\) be a graph with a list assignment \(\ListAsgn{L}\).
	For a subgraph \(G'\) of \(G\) and a color \(c \in \Nat\), denote by \(G'(c)\)
	the induced subgraph of \(G'\) on those vertices whose lists contain
	the color \(c\). We shall denote (maximal connected) components of \(G'(c)\)
	by \(\alpha\), \(\beta\), etc.
	
	\item Let \(P\) be a nonempty path or cycle graph. A vertex \(v \in V(P)\) is an \defining{end point} of \(P\)
	if \(\degree(v) = 1\), and it is an \defining{interior point} of \(P\) if \(\degree(v) = 2\).
	
	\item Let \(P'\) be a nonempty proper connected subgraph of \(P\) (so \(P'\) is itself a nonempty path graph).
	We denote by \(v(P')\) an end point of \(P'\), and by \(w(P')\) a vertex in
	\(P \setminus P'\) that is adjacent to \(v(P')\) (when it exists).
	\end{itemize}
\end{notation}

Next, we record the following simple observation:
\begin{observation}\label{O:main}
	Suppose \(P\) is a nonempty path or cycle graph,
	\(\ListAsgn{L}\) is a \(k\)-list assignment on \(P\) (\(k \geq 1\)), and \(c \in \Nat\) is any color
	such that \(P(c)\) is a nonempty proper subgraph of \(P\). Let \(\alpha\) be a component of \(P(c)\).
	Let \(v(\alpha)\) be an end point of \(\alpha\) for which \(w(\alpha)\) exists.
	Then, there exists a color \(d \in \List{L}{w(\alpha)} \setminus \List{L}{v(\alpha)}\), since
	\(c \not\in \List{L}{w(\alpha)}\) and \(\card{L_{v(\alpha)}} = \card{L_{w(\alpha)}} = k \geq 1\).
	In particular, there exists a component \(\beta\) of \(P(d)\) containing
	\(w(\alpha)\) but not \(v(\alpha)\).
\end{observation}

The following key lemma will allow us to color the alternate
vertices of a path---equipped with a list assignment---in
such a way that the remaining vertices
lose at most one color from each of their lists.

\begin{lemma}\label{L:KeyLemma}
	Let \(P = (V, E)\) be a path graph on at least \(2\) vertices,
	and \(\ListAsgn{L}\) a \(k\)-list assignment on \(P\), \(k \geq 1\).
	Let \(I_{1} \cup I_{2}\) be the unique partition of \(V(P)\)
	into two independent sets. Then, there exists a coloring of \(I_{1}\)
	such that every \(v \in I_{2}\) loses at most one color from its list.
\end{lemma}
\begin{proof}
	Fix an end point \(v(P)\) of \(P\). Choose a color \(c \in \List{L}{v(P)}\)
	and consider the component \(\alpha\) of \(P(c)\) containing \(v(P)\).
	Set \(\Color(v) = c\) for every \(v \in I_{1} \cap V(\alpha)\).
	
	If \(\alpha = P\), then we are done, since every \(v \in I_{2}\)
	has lost at most one color from its list, namely \(c\).
	So, suppose that \(\alpha \neq P\), and define \(P_{1} \defn P - \alpha\).
	Then, \(P_{1}\) is a nonempty path, and there is an end point \(v(\alpha)\)
	of \(\alpha\) such that \(w(\alpha)\) is an end point of \(P_{1}\).
	Choose \(c_{1} \in \List{L}{w(\alpha)} \setminus \List{L}{v(\alpha)}\), which
	exists by \cref{O:main}.
	Let \(\alpha_{1}\) be the component	of \(P_{1}(c_{1})\) containing \(w(\alpha)\),
	and set \(\Color(v) = c_{1}\) for every \(v \in I_{1} \cap V(\alpha_{1})\).
	
	Again, if \(\alpha_{1} = P_{1}\), then we are done. Else, let \(P_{2} \defn P_{1} \setminus \alpha_{1}\),
	and proceed inductively until the process terminates.
\end{proof}

\subsection*{Proof of \texorpdfstring{\cref{T:Main}}{Theorem~1.2}}
\begin{proof}
	Let us first show that any simple \(6\)-regular triangulation
	\(G\) that is \(3\)-chromatic is \(5\)-choosable.

	By \cref{T:Altshuler}, \(G\) is isomorphic to \(T(r, s, t)\)
	for some integers \(r \geq 1\), \(s \geq 1\), and \(0 \leq t < s\).
	It is straightforward to check that when
	\(T(r, s, t)\) is a simple graph, it is \(3\)-chromatic if and only if \(s \equiv 0 \equiv r - t \pmod 3\).
	Moreover, \(T(r, s, t)\) is uniquely \(3\)-colorable whenever this happens,
	so let \(I_{1}, I_{2}, I_{3}\) be the three independent sets defined
	by any coloring of \(T(r, s, t)\) with \(3\) colors. Without loss of generality, we
	fix \(I_{j}\) to be the independent set containing \((1, j)\) for \(j = 1, 2, 3\).

	Consider the subgraph \(G_{1} \defn G - I_{1}\) of \(G\). Note that \(G_{1}\) is
	a \(3\)-regular bipartite graph.
	Orient the edges of \(G_{1}\) as follows: every horizontal edge is directed east,
	every vertical edge is directed north, and every diagonal edge is directed south-east.
	(Figure~\ref{F:G1-orientation} shows this for \(G = T(5, 6, 2)\).)
	More formally, give the orientations as follows (recall that addition
	in the second coordinate is taken modulo \(s\)). For every \(1 \leq i \leq r\) and \(1 \leq j \leq s\)
	such that \(j - i \not\equiv 0 \pmod 3\):
	\begin{itemize}
		\item if \(j - i\equiv 2 \pmod 3\), assign \(\diredge{(i, j)}{(i + 1, j)}\) for all \(1 \leq i < r\),
		and \(\diredge{(r, j)}{(1, j - t)}\);

		\item if \(j - i \equiv 1 \pmod 3\), assign \(\diredge{(i, j)}{(i, j + 1)}\) for all \(1 \leq i \leq r\),
		\(\diredge{(i, j)}{(i + 1, j - 1)}\) for all \(1 \leq i < r\), and \(\diredge{(r, j)}{(1, j - t - 1)}\).
	\end{itemize}
	Then, \(\outdegree_{G_{1}}(v) = 2\) for every \(v \in I_{2}\), and
	\(\outdegree_{G_{1}}(v) = 1\) for every \(v \in I_{3}\).
	
	We claim that \(G_{1}\) with this orientation
	has no odd Eulerian subgraphs. For, if \(H\) is an Eulerian
	subgraph of \(G_{1}\) and \(v \in V(H)\) is not an isolated vertex, then \(\degree_{H}(v)\)
	is a positive even integer, so \(\degree_{H}(v) = 2\).
	Thus, \(H\) is a disjoint union of cycles and isolated vertices, but
	every cycle in a bipartite graph is even, so this proves our claim.
	Since the empty graph is an even Eulerian subgraph of \(G_{1}\),
	by \cref{T:Alon--Tarsi} \(G_{1}\) is \(\ListAsgn{L}\)-choosable
	for every list assignment \(\ListAsgn{L}\)
	such that \(\card{\List{L}{v}} \geq 3\) for all \(v \in I_{2}\) and \(\card{\List{L}{v}} \geq 2\) for all
	\(v \in I_{3}\).
	
	\begin{figure}
	\begin{minipage}{0.5\textwidth}
	\begin{tikzpicture}[font=\scriptsize]
	\path [draw,postaction={on each segment={mid arrow}}]
		(1,2) -- (1,3) -- (2,3) -- (2,4) -- (3,4) -- (3,5) -- (4,5) -- (4,6) -- (5,6)
		(5,1) -- (6,1) -- (6,2)
		(1,5) -- (1,6) -- (2,6)
		(2,1) -- (3,1) -- (3,2) -- (4,2) -- (4,3) -- (5,3) -- (5,4) -- (6,4) -- (6,5)
		(1,2) -- (2,1)
		(2,3) -- (3,2)
		(3,4) -- (4,3)
		(4,5) -- (5,4)
		(5,6) -- (6,5)
		(1,5) -- (2,4)
		(2,6) -- (3,5)
		(4,2) -- (5,1)
		(5,3) -- (6,2);
	\foreach \x in {1,...,6}
		\foreach \y in {1,...,6}{
		\pgfmathparse{Mod(\y-\x,3)==0?1:0}
		\ifnum\pgfmathresult>0
			\draw[fill=red, opacity=0.2] (\x,\y) circle(2pt);
		\else
			\pgfmathparse{Mod(\y-\x,3)==1?1:0}
			\ifnum\pgfmathresult>0
				\draw[fill=blue] (\x,\y) circle(2pt);
			\else
				\draw[fill=green] (\x,\y) circle(2pt);
			\fi
		\fi
		}
	\foreach \y in {2,3,5,6}{
		\node[anchor=east] at (1,\y) {(1,\y)};
		\pgfmathparse{Mod(\y-5,6)}
		\node[anchor=west] at (6,\pgfmathresult+1) {(1,\y)};
	}
	\foreach \x in {2,3,5}
		\node[anchor=north] at (\x,1) {(\x,1)};
	\foreach \x in {2,4,5}
		\node[anchor=south] at (\x,6) {(\x,6)};
	\end{tikzpicture}
	\caption{\(G_{1}\) with oriented edges}\label{F:G1-orientation}
	\end{minipage}\hfill
	\begin{minipage}{0.5\textwidth}
	\begin{tikzpicture}[font=\scriptsize]
	\draw
		(1,1) -- (1,2) -- (2,2) -- (2,3) -- (3,3) -- (3,4) -- (4,4) -- (4,5) -- (5,5) -- (5,6) -- (6,6)
		(3,1) -- (4,1) -- (4,2) -- (5,2) -- (5,3) -- (6,3) -- (6,4)
		(1,4) -- (1,5) -- (2,5) -- (2,6) -- (3,6);
	\draw[dotted]
		(3, 1) -- (2, 2)
		(4, 2) -- (3, 3)
		(5, 3) -- (4, 4)
		(6, 4) -- (5, 5)
		(2, 3) -- (1, 4)
		(3, 4) -- (2, 5)
		(4, 5) -- (3, 6)
		(5, 2) -- (6, 1);
	\foreach \x in {1,...,6}
		\foreach \y in {1,...,6}{
		\pgfmathparse{Mod(\y-\x,3)==0?1:0}
		\ifnum\pgfmathresult>0
			\draw[fill=red] (\x,\y) circle(2pt);
		\else
			\pgfmathparse{Mod(\y-\x,3)==1?1:0}
			\ifnum\pgfmathresult>0
				\draw[fill=blue] (\x,\y) circle(2pt);
			\else
				\draw[fill=green, opacity=0.2] (\x,\y) circle(2pt);
			\fi
		\fi
		}
	\foreach \y in {1,2,4,5}{
		\node[anchor=east] at (1,\y) {(1,\y)};
		\pgfmathparse{Mod(\y-5,6)}
		\node[anchor=west] at (6,\pgfmathresult+1) {(1,\y)};
	}
	\foreach \x in {3,4}
		\node[anchor=north] at (\x,1) {(\x,1)};
	\foreach \x in {2,3,5}
		\node[anchor=south] at (\x,6) {(\x,6)};
	\end{tikzpicture}
	\caption{The subgraph \(H\)}\label{F:H--T(5,6,2)}
	\end{minipage}
	\end{figure}

	Now, suppose that we are given a \(5\)-list assignment \(\ListAsgn{L}\)
	on \(G\).
	If we arbitrarily assign \(\Color(v) \in \List{L}{v}\) for every
	\(v \in I_{1}\), then---in the worst-case scenario---we are left
	with a \(2\)-list assignment on \(G_{1} = G - I_{1}\).
	However, to apply \cref{T:Alon--Tarsi} on \(G_{1}\) with the orientation
	described above, we need \(3\)-lists on \(I_{2}\). So, let
	us now consider the subgraph \(H\) on \(I_{1} \cup I_{2}\) with only the horizontal
	and vertical edges present (see Figure~\ref{F:H--T(5,6,2)}). We will show that we can color \(I_{1}\)
	carefully in such a way that every vertex in \(I_{2}\) loses at most one color in \(H\);
	thus every vertex in \(I_{2}\) will lose at most two colors in \(G\) (since
	every \(v \in I_{2}\) is adjacent to one more vertex of \(I_{1}\) in \(G\) than in \(H\),
	shown by the dotted edges in Figure~\ref{F:H--T(5,6,2)}). Then, we will
	be done by invoking \cref{T:Alon--Tarsi}.
	
	Note that \(H\) is a disjoint union of even cycles
	(in fact, there are exactly \(\gcd(s, r - t) / 3\) cycles in \(H\)). 
	First, assume for the sake of simplicity that there
	is only one cycle in \(H\) (as in Figure~\ref{F:H--T(5,6,2)}). We eliminate the trivial case:
	if there is a color that belongs to every list
	on \(I_{1}\), then we just color
	\(I_{1}\) with that color, and we are done. Thus, it suffices
	to assume that there is no color common to every list on \(I_{1}\).
	
	Now, suppose there exists a color \(c\) such that
	\(H(c)\) has a component \(\alpha\) of even order.
	Assign \(\Color(v) = c\) for all \(v \in I_{1} \cap V(\alpha)\).
	Observe that:
	\begin{itemize}
		\item Every \(v \in V(\alpha) \cap I_{2}\)
		loses at most one color (namely, \(c\)) from its list.
	
		\item Since \(\alpha\) has even order, it has an endpoint \(v_{j}(\alpha)\) in \(I_{j}\)
		for \(j = 1, 2\).
		
		\item The vertex \(w_{1}(\alpha) \in I_{2}\) has not lost any color in its list
		due to the coloring of \(v_{1}(\alpha)\), so we are free to put any color
		on the other vertex adjacent to \(w_{1}(\alpha)\). 
		
		\item On the other hand, the vertex \(w_{2}(\alpha) \in I_{1}\) must be colored
		carefully, since \(v_{2}(\alpha) \in I_{2}\) has already lost one color from its list.
	\end{itemize}
	By \cref{O:main}, there exists a color \(d \in \List{L}{w_{2}(\alpha)} \setminus \List{L}{v_{2}(\alpha)}\).
	So, consider the path \(P \defn H - \alpha\). Apply \cref{L:KeyLemma}
	to \(P\) by starting the coloring at \(w_{2}(\alpha)\) with the color \(d\).
	Then, every \(v \in I_{2}\) will indeed have lost at most one color from its list,
	so we are done.
	
	On the other hand, suppose that for every color \(c\), every component
	of \(H(c)\) has odd order. Choose a color \(c\) and a component \(\alpha\)
	of \(H(c)\) such that both ends of \(\alpha\) lie in \(I_{1}\).
	We can always do this because if \(\alpha'\) is a component of \(H(c')\)
	such that both ends of \(\alpha'\) lie in \(I_{2}\), then choose
	\(c \in \List{L}{w(\alpha')} \setminus \List{L}{v(\alpha')}\) for some
	end point \(v(\alpha')\), and let \(\alpha\) be the component of \(H(c)\)
	containing \(w(\alpha')\). Then, both ends of \(\alpha\) lie in \(I_{1}\)
	as required. Assign \(\Color(v) = c\) for all \(v \in I_{1} \cap V(\alpha)\).
	In this case, we have that:
	\begin{itemize}
		\item Every \(v \in I_{2} \cap V(\alpha)\)
		loses at most one color (namely, \(c\)) from its list.
	
		\item Since both ends, \(v_{1}(\alpha)\) as well as \(v_{2}(\alpha)\),
		lie in \(I_{1}\), the vertices \(w_{1}(\alpha)\) and \(w_{2}(\alpha)\)
		in \(I_{2}\) have not yet lost any color from their lists, so we are
		free to put any color on their other neighbors.
	\end{itemize}
	Thus, we consider the path \(P \defn H - \alpha\), and apply \cref{L:KeyLemma}
	to \(P\) starting from any end point with any color. Then, every \(v \in I_{2}\)
	will again have lost at most one color from its list, so we are done.

	If \(H\) consists of more than one cycle, we repeat this process
	for each cycle. Thus, with this coloring scheme for \(I_{1}\),
	we are left with the required list sizes on the vertices of \(G_{1} = G - I_{1}\),
	so we are done by \cref{T:Alon--Tarsi}.
	Thus, we have shown that every simple \(6\)-regular triangulation
	that is \(3\)-chromatic is \(5\)-choosable. 
	
	Now, on the one hand, any simple triangulation requires at least \(3\) colors for a proper coloring.
	On the other hand, \(H(1) = 7\), so every toroidal graph requires
	no more than \(7\) colors for a proper coloring by \cref{T:Heawood}.
	Furthermore, \cref{T:Dirac} shows that
	\(\chr(G) = 7 \iff \ch(G) = 7\) for any toroidal graph \(G\).
	Hence, if \(G\) is a simple triangulation on the torus
	with \(\jump(G) > 2\), then it must be that \(\chr(G) = 3\) and \(\ch(G) = 6\).
	However, we have shown above that this cannot happen if \(G\) is \(6\)-regular.
	Hence, \(\jump(G) \leq 2\) for any simple \(6\)-regular triangulation \(G\) on the torus.
	We have also seen that \(\jump(G) \leq 2\) for any loopless \(6\)-regular
	triangulation \(G\) that contains multiple edges. This completes the proof
	of \cref{T:Main}.
\end{proof}

\section{Asymptotics of the jump function}\label{S:asymptotics}

For the sake of simplicity, we ignore any ceilings and floors in the following proofs.

\subsection*{Proof of \texorpdfstring{\cref{T:asymp1}}{Theorem~1.3}}
\begin{proof}
	Let \(g \geq 1\). For any graph \(G\), \(\ch(G) - \chr(G) \leq \ch(G)\),
	and if \(G\) is embeddable in \(S_{g}\) then \(\ch(G) \leq H(g)\)
	by \cref{T:Heawood}. Thus, \(\jump(g) \leq H(g) \leq 7 \sqrt{g}\) for all \(g \geq 1\).

	To establish the lower bound,
	let \(m \geq 2\) be fixed, and consider the graph \(K_{m * r}\) for \(r \geq 2\).
	Note that \(\chr(K_{m * r}) = \chr(K_{r}) = r\), and
	\(\ch(K_{m * r}) \geq c_{1} r \log(m)\)
	by \cref{T:Alon-multipartite}.
	Hence, \(\jump(K_{m * r}) \geq c_{1} r \log(m) - r\).
	Pick \(m\) large enough so that \(c_{1} \log(m) \geq 2\).
	Then, \(\jump(K_{m * r}) \geq r\). Thus, we would like to show that \(K_{m * r}\)
	is embeddable in \(S_{g}\) for \(g \leq \bigoh{r^{2}}\).
	
	\cref{T:Ringel-Youngs} says that \(K_{r}\) is embeddable in \(S_{\OGenus(r)}\),
	so start with an embedding of \(K_{r}\) into \(S_{\OGenus(r)}\) and then
	add handles to \(S_{\OGenus(r)}\) to accommodate
	the extra edges of \(K_{m * r}\). Since \(\card{E(K_{r})} = \binom{r}{2}\) and
	\(\card{E(K_{m * r})} = m^{2}\binom{r}{2}\),
	we need to add at most \((m^{2} - 1) \binom{r}{2}\) handles to \(S_{\OGenus(r)}\).
	Since \(\OGenus(r) + (m^{2} - 1)\binom{r}{2} \leq m^{2}r^{2}/2\), \(K_{m * r}\) is embeddable
	in \(S_{g(r)}\) where \(g(r) \defn m^{2} r^{2} / 2\), and this is what we wanted to show.
	Hence, \(\jump\parens[\big]{g(r)} \geq r = c \sqrt{g(r)}\) for all \(r \geq 2\), where \(c \defn \sqrt{2}/m\).
	
	Finally, suppose that \(g(r) \leq g' \leq g(r+1)\) for a fixed \(r \geq 2\).
	\(K_{m * r}\) is embeddable in \(S_{g'}\) as well, so
	\[
		\jump(g') \geq \jump(K_{m * r}) \geq r = c \sqrt{g(r + 1)} - 1 \geq c \sqrt{g'} - 1 \geq c' \sqrt{g'}
	\]
	for any positive constant \(c' < c\), provided \(r\) (and hence \(g(r)\))
	is sufficiently large. This completes the proof.
\end{proof}

\subsection*{Proof of \texorpdfstring{\cref{T:asymp2}}{Theorem~1.6}}
\begin{proof}
	Let \(\epsilon > 0\) and \(g \geq 1\). Let \(G\) be simple and embeddable in \(S_{g}\)
	with \(\chr(G) = r\), and let \(\card{V(G)} = m\). To show
	that \(\jump(G) \leq \epsilon\sqrt{g}\) it suffices to show
	that \(\ch(G) \leq \epsilon\sqrt{g}\) since \(\jump(G) \leq \ch(G)\).
	By \cref{T:core}, there is no loss of generality in assuming
	that \(G\) is equal to its \((\epsilon\sqrt{g})\)-core.
	Thus, the minimum degree, and hence average degree \(2\nume / \numv\), of \(G\)
	is bounded below by \(\epsilon \sqrt{g}\).
	By Euler's formula, \(2\nume / \numv \leq 6 + 12(g - 1)/m\),
	and hence \(m \leq 24\sqrt{g}/\epsilon\).

	Now, \(G\) is a subgraph of \(K_{m * r}\), and
	\(\ch(G) \leq \ch(K_{m * r}) \leq c_{2} r \log(m)\)
	by \cref{T:Alon-multipartite}.
	Then, for every \(\delta > 0\),
	\[
		\ch(K_{m * r})
			\leq c_{2} r \log(m)
			\leq c_{2} \parens*{\frac{\delta \sqrt{g}}{\log(g)}} \parens*{\log(24/\epsilon) + \frac{1}{2}\log(g)}
			\leq c_{2} \delta \sqrt{g} \parens*{\frac{\log(24/\epsilon)}{\log(g)} + \frac{1}{2}}
			< c_{2} \delta\sqrt{g}
	\]
	for all sufficiently large \(g\). Let \(\delta = \epsilon/c_{2}\)
	to finish the proof.
\end{proof}

\section{Conclusions and further remarks}\label{S:conclusion}

\subsection{The jump for general toroidal graphs}

We discuss some partial results towards computing \(\jump(G)\)
for a general toroidal graph \(G\).

Firstly, we will need the following result of Alon and Tarsi~\cite{AloTar92}
used in their proof that bipartite planar graphs
are \(3\)-choosable. For a graph \(G\), define
\(L(G) \defn \max_{H \leq G}\{ \card{E(H)} / \card{V(H)} \}\), where the maximum
is taken over all subgraphs \(H = (V, E)\) of \(G\).

\begin{theorem}[Alon--Tarsi~\cite{AloTar92}, 1992]
Every bipartite graph \(G\) is 	\((\ceil{L(G)} + 1)\)-choosable.
\end{theorem}

Using Euler's formula, it follows that any planar bipartite graph \(G = (V, E)\) satisfies
\(\nume \leq 2 \numv - 4\), so \(L(G) \leq 2\).
Hence, every bipartite planar graph is \(3\)-choosable.

The same analysis goes through for toroidal bipartite graphs, since
Euler's formula applied on toroidal bipartite graphs \(G = (V, E)\)
yields \(\nume \leq 2 \numv\),
so again \(L(G) \leq 2\). Thus, every bipartite toroidal graph is also \(3\)-choosable.
Hence, if \(G\) is a toroidal graph with \(\jump(G) > 2\), then
\(G\) cannot be bipartite, so it must be that \(\chr(G) \geq 3\).
On the other hand, if \(\chr(G) \geq 4\), then \(\jump(G) \leq 2\),
since every toroidal graph is \(7\)-choosable, but every
\(7\)-list-chromatic toroidal graph is also \(7\)-chromatic,
by \cref{T:Dirac}.

Hence, any counterexample to the claim that \(\jump(1) = 2\)
must be a toroidal graph \(G\) for which \(\chr(G) = 3\) and \(\ch(G) = 6\).
Now, suppose that there do exist such graphs,
and choose one with the minimal number of vertices.
Then, by \cref{T:core}, its minimum degree is at least \(5\).
Recall that Euler's formula shows that the average degree, \(2\nume/\numv\),
of a toroidal graph \(G\) satisfies \(2\nume/\numv \leq 6\), and equality holds
if \(G\) is a triangulation.
Thus, if \(\delta(G) = 6\), then \(G\) is in fact a \(6\)-regular triangulation.
This motivates one to examine \(6\)-regular triangulations in particular,
and we have shown in \cref{T:Main} that \(\jump(G) \leq 2\) for such graphs.

Thus, any minimal counterexample \(G\) must satisfy \(\delta(G) = 5\).
Moreover, if we add edges to \(G\) while preserving its toroidicity and chromatic number,
then we can get a minimal counterexample whose faces are all either triangular or
quadrangular, which we call a \defining{mosaic} following Nakamoto, Noguchi and Ozeki~\cite{NakNogOze19}.
Thus, it follows that \(\jump(1) > 2\) if and only if there is a
\(3\)-chromatic toroidal mosaic of minimum degree \(5\) that is not \(5\)-choosable.

Now, it so happens that there do exist \(3\)-chromatic
toroidal graphs \(G\) with \(\delta(G) = 5\) (see \cref{F:lemma1,F:lemma2,F:lemma3}; again,
the edges between the top and bottom rows are omitted).
While ad hoc arguments can show that \(\jump(G) \leq 2\) for some
of these graphs, a unified argument is still missing.
An Alon--Tarsi type argument appears more difficult to implement since there maybe several
vertices with degree greater than \(10\), so it is not clear how to orient the edges appropriately.

\afterpage{
\begin{figure}
\vspace*{0.5cm}
\begin{minipage}{0.5\textwidth}
	\begin{tikzpicture}[font=\scriptsize]
	\foreach \y in {2,...,6}
		\foreach \x in {1,4,5}{
			\draw
			(\x,\y) -- (\x+1,\y-1);
	}
	\foreach \y in {2,...,6}
		\foreach \x in {1,2,4,5,6}{
			\draw
			(\x,\y) -- (\x,\y-1);
	}
	\foreach \x in {1,...,5}
		\foreach \y in {1,2,3,5,6}{
			\draw
			(\x,\y) -- (\x+1,\y);
	}
	\foreach \y in {2,3,6}{
		\draw
		(2,\y) -- (3,\y-1)
		(3,\y) -- (4,\y-1)
		(3,\y) -- (3,\y-1);
	}
	\foreach \x in {1,4,5}{
		\draw
		(\x,4) -- (\x+1,4);
	}
	\draw
	(3,3) -- (2,4) -- (4,4) -- (3,5);
	\foreach \x in {1,...,6}{
		\foreach \y in {1,2,3,5,6}{
			\pgfmathparse{Mod(\y-\x,3)==0?1:0}
			\ifnum\pgfmathresult>0
				\draw[fill=red] (\x,\y) circle(2pt);
			\else
				\pgfmathparse{Mod(\y-\x,3)==1?1:0}
				\ifnum\pgfmathresult>0
					\draw[fill=blue] (\x,\y) circle(2pt);
				\else
					\draw[fill=green] (\x,\y) circle(2pt);
				\fi
			\fi
		}
	}
	\foreach \x in {1,2,4,5,6}{
		\pgfmathparse{Mod(4-\x,3)==0?1:0}
		\ifnum\pgfmathresult>0
			\draw[fill=red] (\x,4) circle(2pt);
		\else
			\pgfmathparse{Mod(4-\x,3)==1?1:0}
			\ifnum\pgfmathresult>0
				\draw[fill=blue] (\x,4) circle(2pt);
			\else
				\draw[fill=green] (\x,4) circle(2pt);
			\fi
		\fi		
	}
	\foreach \y in {1,...,6}{
		\node[anchor=east] at (1,\y) {(1,\y)};
		\pgfmathparse{Mod(\y-5,6)}
		\node[anchor=west] at (6,\pgfmathresult+1) {(1,\y)};
	}
	\foreach \x in {2,...,5}{
		\node[anchor=north] at (\x,1) {(\x,1)};
		\node[anchor=south] at (\x,6) {(\x,6)};
	}
	\end{tikzpicture}
	\caption{Toroidal \(G\) with \(\chr(G) = 3\),
	\(\delta(G)=5\), \(\Delta(G)=6\)}\label{F:lemma1}
\end{minipage}\hfill
\begin{minipage}{0.5\textwidth}
	\begin{tikzpicture}[font=\scriptsize]
	\foreach \y in {2,...,6}
		\foreach \x in {1,4,5}{
			\draw
			(\x,\y) -- (\x+1,\y-1);
	}
	\foreach \y in {2,...,6}
		\foreach \x in {1,2,4,5,6}{
			\draw
			(\x,\y) -- (\x,\y-1);
	}
	\foreach \x in {1,...,5}
		\foreach \y in {1,2,5,6}{
			\draw
			(\x,\y) -- (\x+1,\y);
	}
	\foreach \y in {2,6}{
		\draw
		(2,\y) -- (3,\y-1)
		(3,\y) -- (4,\y-1)
		(3,\y) -- (3,\y-1);
	}
	\foreach \x in {1,4,5}{
		\draw
		(\x,3) -- (\x+1,3)
		(\x,4) -- (\x+1,4);
	}
	\draw
	(3,2) -- (2,3) -- (4,3)
	(3,5) -- (4,4) -- (2,4)
	(4,4) -- (2,3);
	\foreach \y in {1,...,6}{
		\foreach \x in {1,2,4,5,6}{
			\pgfmathparse{Mod(\y-\x,3)==0?1:0}
			\ifnum\pgfmathresult>0
				\draw[fill=red] (\x,\y) circle(2pt);
			\else
				\pgfmathparse{Mod(\y-\x,3)==1?1:0}
				\ifnum\pgfmathresult>0
					\draw[fill=blue] (\x,\y) circle(2pt);
				\else
					\draw[fill=green] (\x,\y) circle(2pt);
				\fi
			\fi
		}
	}
	\foreach \y in {1,2,5,6}{
		\pgfmathparse{Mod(\y-3,3)==0?1:0}
		\ifnum\pgfmathresult>0
			\draw[fill=red] (3,\y) circle(2pt);
		\else
			\pgfmathparse{Mod(\y-3,3)==1?1:0}
			\ifnum\pgfmathresult>0
				\draw[fill=blue] (3,\y) circle(2pt);
			\else
				\draw[fill=green] (3,\y) circle(2pt);
			\fi
		\fi		
	}
	\foreach \y in {1,...,6}{
		\node[anchor=east] at (1,\y) {(1,\y)};
		\pgfmathparse{Mod(\y-5,6)}
		\node[anchor=west] at (6,\pgfmathresult+1) {(1,\y)};
	}
	\foreach \x in {2,...,5}{
		\node[anchor=north] at (\x,1) {(\x,1)};
		\node[anchor=south] at (\x,6) {(\x,6)};
	}
	\end{tikzpicture}
	\caption{Toroidal \(G\) with \(\chr(G) = 3\),
	\(\delta(G)=5\), \(\Delta(G)=7\)}\label{F:lemma2}
\end{minipage}\hfill
\end{figure}

\begin{figure}
	\vspace*{1cm}
	\begin{tikzpicture}[font=\scriptsize]
	\foreach \y in {2,...,6}{
		\pgfmathsetmacro{\z}{7-\y}
		\foreach \x in {1,...,\z}{
			\draw
			(\x,\y) -- (\x,\y-1)
			(\x,\y-1) -- (\x+1,\y-1)
			(\x,\y) -- (\x+1,\y-1);
		}
	}
	\foreach \x in {4,...,8}{
		\pgfmathsetmacro{\z}{13-\x}
		\foreach \y in {\z,...,9}{
			\draw
			(\x+1,\y) -- (\x+1,\y-1)
			(\x,\y) -- (\x+1,\y)
			(\x,\y) -- (\x+1,\y-1);
		}
	}
	\foreach \x in {6,7,8}{
		\draw
		(\x,2) -- (\x,1)
		(\x,1) -- (\x+1,1)
		(\x,2) -- (\x+1,2)
		(\x,2) -- (\x+1,1);
	}
	\foreach \x in {1,2,3}{
		\draw
		(\x,9) -- (\x,8)
		(\x,8) -- (\x+1,8)
		(\x,9) -- (\x+1,9)
		(\x,9) -- (\x+1,8);
	}
	\foreach \y in {3,4}{
		\draw
		(8,\y) -- (9,\y)
		(8,\y) -- (8,\y-1)
		(9,\y) -- (9,\y-1)
		(8,\y) -- (9,\y-1);
	}
	\foreach \y in {7,8}{
		\draw
		(1,\y-1) -- (2,\y-1)
		(1,\y) -- (1,\y-1)
		(2,\y) -- (2,\y-1)
		(1,\y) -- (2,\y-1);
	}
	\draw
	(4,9) -- (4,8) -- (5,8)
	(2,5) -- (2,6)
	(5,2) -- (6,2)
	(9,1) -- (9,2)
	(8,4) -- (8,5);
	\draw
	(3,8) -- (5,5) -- (8,3)
	(4,8) -- (5,5) -- (8,4)
	(2,7) -- (5,5) -- (7,2)
	(2,6) -- (5,5) -- (6,2)
	(3,4) -- (5,5) -- (6,7)
	(4,3) -- (5,5) -- (7,6);
	\foreach \y in {1,...,6}{
		\pgfmathsetmacro{\z}{7-\y}
		\foreach \x in {1,...,\z}{
			\pgfmathparse{Mod(\y-\x,3)==0?1:0}
			\ifnum\pgfmathresult>0
				\draw[fill=red] (\x,\y) circle(2pt);
			\else
				\pgfmathparse{Mod(\y-\x,3)==1?1:0}
				\ifnum\pgfmathresult>0
					\draw[fill=blue] (\x,\y) circle(2pt);
				\else
					\draw[fill=green] (\x,\y) circle(2pt);
				\fi
			\fi
		}
	}
	\foreach \x in {4,...,9}{
		\pgfmathsetmacro{\z}{13-\x}
		\foreach \y in {\z,...,9}{
			\pgfmathparse{Mod(\y-\x,3)==0?1:0}
			\ifnum\pgfmathresult>0
				\draw[fill=red] (\x,\y) circle(2pt);
			\else
				\pgfmathparse{Mod(\y-\x,3)==1?1:0}
				\ifnum\pgfmathresult>0
					\draw[fill=blue] (\x,\y) circle(2pt);
				\else
					\draw[fill=green] (\x,\y) circle(2pt);
				\fi
			\fi
		}
	}
	\draw[fill=red]
	(1,7) circle(2pt)
	(2,8) circle(2pt)
	(3,9) circle(2pt)
	(7,1) circle(2pt)
	(8,2) circle(2pt)
	(9,3) circle(2pt)
	(5,5) circle(2pt);
	\draw[fill=green]
	(2,7) circle(2pt)
	(3,8) circle(2pt)
	(1,9) circle(2pt)
	(6,2) circle(2pt)
	(8,4) circle(2pt)
	(8,1) circle(2pt)
	(9,2) circle(2pt);
	\draw[fill=blue]
	(1,8) circle(2pt)
	(2,9) circle(2pt)
	(2,6) circle(2pt)
	(4,8) circle(2pt)
	(7,2) circle(2pt)
	(8,3) circle(2pt)
	(9,1) circle(2pt);
	\foreach \y in {1,...,9}{
		\node[anchor=east] at (1,\y) {(1,\y)};
		\pgfmathparse{Mod(\y-2,9)}
		\node[anchor=west] at (9,\pgfmathresult+1) {(1,\y)};
	}
	\foreach \x in {2,...,8}{
		\node[anchor=north] at (\x,1) {(\x,1)};
		\node[anchor=south] at (\x,9) {(\x,9)};
	}
	\end{tikzpicture}
	\caption{Toroidal \(G\) with \(\chr(G) = 3\),
	\(\delta(G)=5\), \(\Delta(G)=12\)}\label{F:lemma3}
\end{figure}
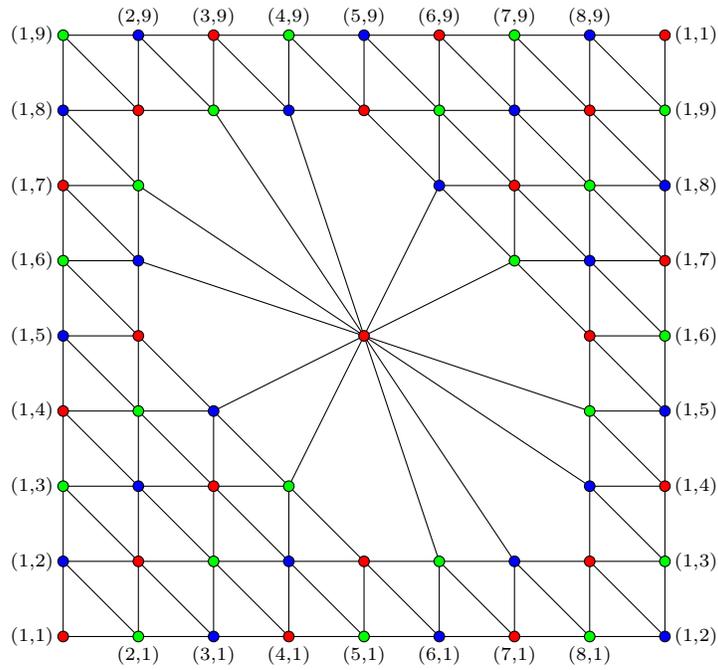
\clearpage
}

A relevant concept to mention here is \defining{color criticality} of graphs.
For \(k \geq 1\), a \defining{\(k\)-critical} graph is one that is not \((k - 1)\)-colorable
but whose proper subgraphs are. Similarly,
a \defining{\(k\)-list-critical} graph
is one for which there is a \(k\)-list assignment \(\ListAsgn{L}\)
such that the graph is not \(\ListAsgn{L}\)-choosable but every proper subgraph
is. Every \(k\)-critical graph is thus also \(k\)-list-critical, but in general
there are \(k\)-list-critical graphs that are not \(k\)-critical. Furthermore,
a \(k\)-critical graph cannot contain another \(k\)-critical graph as a proper subgraph,
but this statement is not true if we replace ``\(k\)-critical'' with ``\(k\)-list-critical'',
since a graph may be list-critical with respect to a list assignment \(\ListAsgn{L}\)
but not \(\ListAsgn{L'}\), yet it may contain a proper subgraph that is list-critical
with respect to \(\ListAsgn{L'}\) but not \(\ListAsgn{L}\). So, one defines
a \defining{minimal} \(k\)-list-critical graph to be one which does not
contain a \(k\)-list-critical graph as a proper subgraph.

Observe that a minimal counterexample to our claim that also has the least number
of edges must be a minimal \(6\)-list-critical graph. Postle and Thomas~\cite{PosTho18} showed
in 2018 that there are only finitely many \(6\)-list-critical graphs
on any surface. Thus, the claim that every toroidal graph has
gap at most \(2\) needs to be verified against
only finitely many potential counterexamples; however, a complete list of
\(6\)-list-critical graphs on the torus does not yet exist in the literature
(though a complete list of \(6\)-critical graphs was given by Thomassen~\cite{Tho94a} in 1994).
Stiebitz, Tuza and Voigt~\cite{StiTuzVoi09} in 2009 showed that 
for all \(2 \leq r \leq k\), there is a minimal \(k\)-list critical graph
that is \(r\)-chromatic, so one also cannot immediately rule out the existence
of a minimal \(6\)-list-critical graph on the torus that is \(3\)-chromatic.

We note that computing tight bounds even for triangulations on surfaces
of higher values of \(g\) appears difficult due to the lack of structural results
similar to Altshuler's theorem on the torus (\cref{T:Altshuler}).

\subsection{Analogous results for graphs embeddable in nonorientable surfaces}

As mentioned in \cref{S:Introduction}, the restriction to orientable surfaces
has only been for convenience, and the results proved in \cref{S:asymptotics} also hold
over nonorientable surfaces with suitable modifications, which we indicate below.

The Heawood number for nonorientable surfaces is
	\[
		\widetilde{H}(k) \defn \floor*{\frac{7+\sqrt{1+24k}}{2}}.
	\]
By essentially the same argument as given by Heawood~\cite{Hea90} for the orientable case,
\(\widetilde{H}(k)\) is an upper bound for the chromatic number of any graph \(G\)
embeddable in \(N_{k}\), \(k \geq 1\), and the argument also carries forward to
the choice number, so \(\chr(G) \leq \ch(G) \leq \widetilde{H}(k)\)
for every \(G\) embeddable in \(N_{k}\), \(k \geq 1\).

This upper bound is not tight for the Klein bottle \(N_{2}\),
as shown by Franklin~\cite{Fra34} in 1934: \(\widetilde{H}(2) = 7\),
but every graph embeddable in \(N_{2}\) is \(6\)-colorable.
In particular, \(K_{6}\) is embeddable in \(N_{2}\),
but \(K_{7}\) is not. Using \cref{T:Brooks}
one can show that every graph embeddable in \(N_{2}\) is also \(6\)-choosable.

For every other nonorientable surface, Heawood's upper bound is tight.
Ringel~\cite{Rin54a} showed in 1954 that for every \(r \geq 1\), except \(r = 2\),
\(K_{r}\) is embeddable in \(N_{\NOGenus(r)}\)
for \(\NOGenus(r) \defn \ceil{(r - 3)(r - 4) / 6}\) and this is best possible. This implies
that \(K_{\widetilde{H}(k)}\) is embeddable in \(N_{k}\) for all \(k \geq 1\) except \(k = 2\).
Dirac's map color theorem~\cites{Dir52c,AlbHut79} also extends to \(N_{k}\) for all \(k \geq 1\), except \(k = 2\).
In the latter case, \(K_{6}\) is not the only \(6\)-chromatic graph that is
embeddable in \(N_{2}\) (see~\cite{AlbHut79} for an example of another such graph).
Dirac's map color theorem for the choice number~\cites{BohMohSti99,KraSkr06}
also extends to \(N_{k}\) for all \(k \geq 1\), except \(k = 2\).

We define, analogously, the functions \(\Jump(k)\) and \(\Jump(k,r)\).
Using the above results it is not hard to show the following:
\begin{theorem}
	\(\Jump(1) = 2\). That is, for graphs \(G\) embeddable in the projective plane \(N_{1}\),
	\(\jump(G) \leq 2\). Moreover, this bound is best possible.
\end{theorem}

The asymptotic bounds on \(\Jump(k)\) and \(\Jump(k, r)\) are also
the same as in the orientable case. The proofs go through in a similar manner.
One just needs to know what the result is of adding a handle to a nonorientable surface,
for which the following theorem is useful:
\begin{theorem}[Dyck~\cite{Dyc88}, 1888]
	The connected sum of a torus and a projective plane is isomorphic to
	the connected sum of three projective planes.
\end{theorem}
One also needs to modify Euler's formula for nonorientable surfaces as follows: if \(G\) is embeddable in
\(N_{k}\), \(k \geq 1\), and \(\numv\), \(\nume\) and \(\numf\) denote the number
of vertices, edges, and faces of \(G\) in an embedding of \(G\) in \(N_{k}\), respectively,
then \(\numv - \nume + \numf \geq 2 - k\), with equality holding if every face
is homeomorphic to a disc.

\subsection{Concluding remarks}

We have proved in \cref{T:Main} that any loopless \(6\)-regular triangulation
\(G\) on the torus satisfies \(\jump(G) \leq 2\). 
We speculate that \emph{every} \(3\)-chromatic toroidal graph is in fact \(5\)-choosable:
\begin{conjecture}\label{Conj:1}
	\(\jump(1, 3) = 2\).
\end{conjecture}
	A resolution to \cref{Conj:1} will also answer \cref{Q:toroidal-bound},
	as we have observed that the maximum gap for toroidal graphs is at least \(2\),
	and neither toroidal bipartite graphs nor toroidal graphs with chromatic number
	at least \(4\) can attain a gap greater than \(2\).
	
	An easy example of a \(6\)-regular triangulation \(G\) on the torus for
	which \(\jump(G) = 1\) is furnished by \(T(3, 3, 0)\),
	which is isomorphic to the complete multipartite
	graph \(K_{3 * 3}\). Kierstead~\cite{Kie00} in 2000 showed that \(\ch(K_{3*r}) = \ceil{(4r - 1)/3}\)
	for every \(r \geq 1\), so \(\ch(K_{3 * 3}) = 4\), and since
	\(T(3, 3, 0) \equiv T(r, s, t)\) satisfies \(s \equiv 0 \equiv r - t \pmod 3\),
	it is \(3\)-chromatic. Hence, \(\jump\parens[\big]{T(3, 3, 0)} = 1\).
	
	We are also able to show that the \(3\)-chromatic \(6\)-regular triangulations
	are not \(3\)-choosable, which will appear in a forthcoming paper~\cite{BalSan21b}.
	But, we do not have an explicit example of a \(6\)-regular triangulation for which
	\(\jump(G) = 2\). One of the authors is unsure whether any such graph exists,
	while the other believes that there might be one; in any case, we pose the following question:
\begin{question}\label{Q:2}
	Are there \(3\)-chromatic \(6\)-regular triangulations that
	are \(5\)-list-chromatic?
\end{question}

The results on the asymptotic behavior of \(\jump(g)\) and \(\jump(g,r)\)
motivate the following conjectures that refine the results proved
in \cref{S:asymptotics}:

\begin{conjecture}\label{Conj:3}
	\(\jump(g, r)\) is unimodal in \(r\) for each fixed \(g\). That is,
	there exists \(r_{0} \equiv r_{0}(g)\) such that
	\[
		\jump(g, 1) \leq \jump(g, 2) \leq \dotsb \leq \jump(g, r_{0}) \geq \dotsb \geq \jump\parens[\big]{g, H(g)}.
	\]
\end{conjecture}

This is already seen to be true for planar and toroidal graphs.
Firstly, a loopless graph is \(1\)-chromatic if and only if it is empty,
which implies that it is also \(1\)-list-chromatic, so \(\jump(g, 1) = 0\)
for all \(g \geq 0\). Next, the results mentioned in \cref{S:Introduction}
show that
\(\jump(0, 2) = 1\), 
\(\jump(0, 3) = 2\), 
and \(\jump(0, 4) = 1\). 
Lastly, \(\jump(0, r) = 0\) for \(r \geq 5\) by the four color theorem
due to Appel and Haken~\cites{AppHak77,AppHakKoc77}.
So, \(\jump(0, r)\) is indeed unimodal in \(r\).

Similarly, for toroidal graphs, though we don't have precise values
of \(\jump(g, r)\) for all \(r\), the discussion so far shows that
\(\jump(1, 2) = 1\), \(\jump(1, 3) = 2\) or \(3\), \(\jump(1, 4) = 1\) or \(2\),
\(\jump(1, 5) = 0\) or \(1\), \(\jump(1, 6) = 0\) and \(\jump(0, 7) = 0\).
So, \(\jump(1, r)\) is unimodal in \(r\) as well.

It is not hard to show that unimodality also holds for \(\Jump(1, r)\)
and \(\Jump(2, r)\),
that is, for the projective plane and the Klein bottle; so, \cref{Conj:3} can be extended
to the nonorientable case as well.

One also notices that there may be more than one value \(r_{0}\)
at which the maximum gap is attained. So, define
\(r_{\max} \equiv r_{\max}(g)\) to be the least value of \(r_{0}(g)\)
in \cref{Conj:3}. We conjecture that \(\sqrt{g}\) is the correct order of \(r\)
at which \(\jump(g, r)\) attains its maximum value for each fixed \(g\):
\begin{conjecture}\label{Conj:4}
	For all sufficiently large \(g\), \(r_{\max} = \bigtheta{\sqrt{g}}\).
\end{conjecture}

Again, \cref{Conj:4} can be extended analogously to the nonorientable case as well.

Finally, a structural result on the Klein bottle \(N_{2}\) for \(6\)-regular triangulations,
similar in spirit to Altshuler's theorem on the torus (\cref{T:Altshuler}),
was given independently by Negami~\cite{Neg84} in 1984 and Thomassen~\cite{Tho91} in 1991.
We pose the following question on the Klein bottle analogous to \cref{Q:toroidal-bound} for the torus,
and also ask whether the \(6\)-regular triangulations on the Klein bottle can be examined to get a
result analogous to \cref{T:Main} for the torus.
\begin{question}\label{Q:5}
	What is \(\Jump(1)\)? That is, how large can the gap between the choice
	number and chromatic number be for a graph embeddable on the Klein bottle?
\end{question}

Just as in the toroidal case, it is not hard to show that the maximum gap cannot be smaller than \(2\),
and we ask how large this gap can get.

\begin{question}\label{Q:6}
	What is the maximum value of \(\Jump(G)\) for any loopless \(6\)-regular
	triangulation \(G\) on the Klein bottle?
\end{question}

However, computing \(\Jump(k)\) precisely for higher values of \(k\) seems difficult
for the same reason as in the orientable case.

\newpage

\end{document}